\documentclass[11pt]{amsart}
\usepackage[colorlinks, linkcolor=blue,  citecolor=blue, urlcolor=blue]{hyperref}%
\usepackage{url}
\usepackage{xspace}
\usepackage{verbatim}
\usepackage{alltt}
\usepackage{graphicx}
\usepackage{subfigure}
\usepackage{enumerate}
\graphicspath{{figures/}}
\usepackage{pdfsync}

\newcommand{\OS}{X}
\newcommand{\OT}{Y}
\newcommand{\eps}{\epsilon}
\newcommand{\Fm}{F_M}
\newcommand{\Fp}{F_P^\epsilon}
\newcommand{\Fe}{F^\epsilon}
\newcommand{\Fa}{F_A}
\newcommand{\Dt}{\mathcal{D}}
\newcommand{\G}{\mathcal{G}}

\usepackage{dsfont}

\newcommand{\da}{d\alpha}

\newcommand{\MA}{{Monge-Amp\`ere}\xspace}
\newcommand{\M}{\mathbb{M}}
\newcommand{\Hd}{H}

\newtheorem{theorem}{Theorem}[section]
\newtheorem*{theoremx}{Theorem}

\newtheorem{lemma}[theorem]{Lemma}

\newtheorem{definition}[theorem]{Definition}
\theoremstyle{remark}
\newtheorem*{remark}{Remark}
\newtheorem{example}[theorem]{Example}

\newcommand{\norm}[1]{\Vert#1\Vert}

\DeclareMathOperator{\trace}{trace}

\DeclareMathOperator{\dist}{dist}
\newcommand{\grad}{\nabla}
\newcommand{\bq}{\begin{equation}}
\newcommand{\eq}{\end{equation}}
\newcommand{\R}{\mathbb{R}}
\newcommand{\Rd}{\R^d}
\newcommand{\Rn}{\R^n}
\newcommand{\e}{\epsilon}
\newcommand{\bO}{\mathcal{O}}

\newcommand{\blue}[1]{{\color{blue}{#1}}}

\graphicspath{{../figures/}}

\begin{document}

\title[Viscosity Solutions for {M}onge-{A}mp\`ere]{A viscosity solution approach to the {M}onge-{A}mp\`ere formulation of the Optimal Transportation Problem}

\author{Jean-David Benamou}
\address{INRIA, Domaine de Voluceau, BP 153 le Chesnay Cedex, France}
\email{jean-david.benamou@inria.fr}
\author{Brittany D. Froese}
\address{Department of Mathematics, University of Texas at Austin, 2515 Speedway, 
Austin, TX 78712}
\email{bfroese@math.utexas.edu}

\author{Adam M. Oberman}
\address{Department of Mathematics, McGill University, 805 Sherbrooke Street West
Montreal, Quebec H3A0B9
Canada}
\email{adam.oberman@mcgill.ca}
\urladdr{www.sfu.ca/~aoberman}

%\author[J.-D. Benamou \and B. D. Froese \and A. M. Oberman]{ \and  \and Adam M. Oberman}
%
%\address{{Department of Mathematics, Simon Fraser University\\ Burnaby, British Columbia, Canada, V5A 1S6}}
%\email{aoberman,bdf1@sfu.ca}

\date{\today}

\begin{abstract} 
In this work we present a numerical method for the Optimal Mass Transportation problem.  Optimal Mass Transportation (OT) is an active  research field in mathematics.
 It has recently led to significant theoretical results as well as applications in diverse areas.  Numerical solution techniques for the OT problem remain underdeveloped. 
% Building on previous work,~\cite{ObermanFroeseMATheory, ObermanFroeseFiltered}, we introduce in this article 
The solution is obtained by solving  the second boundary value problem for the  \MA equation, a fully nonlinear elliptic partial differential equation (PDE). 
%We first introduce a tractable  equivalent formulation of 
Instead of  standard boundary conditions the problem has global state constraints.  These are reformulated as a tractable local PDE.
We give a proof of convergence of the numerical method, using the theory of viscosity solutions. Details of the implementation and a fast solution method are provided in the companion paper \cite{SBVP_Num}.  \end{abstract}


\maketitle

%\tableofcontents
% !TEX root = SBVPmain.tex

\section{Introduction}

The Optimal Transportation (OT) problem is a simply posed mathematical problem which dates back more than two centuries.  It has led to significant results in probability, analysis, and  Partial Differential Equations (PDEs), among other areas.  
 The core theory is by now well established:  good presentations are available  the textbook~\cite{Villani} and the survey~\cite{EvansSurvey}.  The subject continues to find new relevance,  in particular to geometry and curvature (see~\cite{VillaniBook2}), and to dissipative evolutionary equations (see~\cite{MR1617171,MR1842429} and the resulting literature).   Application areas include image registration, mesh generation,  reflector design,  astrophysics, and meteorology (see~\cite{Villani} for references).  

This article presents a numerical method for the Optimal Transportation (OT) problem with quadratic costs.  
It is formally equivalent to the second boundary value problem for the \MA equation, a fully nonlinear elliptic partial differential equation (PDE). 
In this setting the equation lacks standard boundary conditions: 
instead there are state constraints on the gradient of the solution.  Because these state constraints are not easily enforced, we reformulate them using the distance function to the target set, which results in a local but implicit Hamilton-Jacobi PDE at the boundary.
Our  main result is a proof of convergence of solutions of the finite difference approximation to the PDE, using the theory of viscosity solutions.

%, building on previous work,~\cite{ObermanFroeseMATheory, ObermanFroeseFiltered}.  

Going from a theoretically convergent algorithm to the simplest and most efficient numerical implementation requires  more explanation, a different emphasis from the mathematical aspects of the convergence proof.   For this reason, we present the details of the implementation in a companion paper,~\cite{SBVP_Num}.   
%In particular, in practice, a compact stencil can be used, even thought the theoretical convergence proof requires a wide stencil.  
Basic computational results are presented here as an illustration.   Additional computational results, including some which are not covered by the theory (e.g. non-convex targets, Alexsandroff solutions), can be found in \cite{SBVP_Num}.

\subsection{Optimal Transportation and the \MA PDE}
\label{mak}
%Our interest  in the \MA PDE is linked  to the Monge-Kantorovitch or
The OT problem is described as follows.  Suppose we are given two probability densities, where 
\begin{equation} 
\label{DensityAss}
\begin{aligned}
\rho_\OT(y) &\text{ is Lipschitz continuous with convex support } \OT
\\
\rho_\OS(x) &\text{ is bounded with support } \OS,
\\
\end{aligned}
%\begin{aligned}
%\\
%\rho_\OT, &\text{ a probability density supported on } \OT
%\end{aligned}
\end{equation} 
and $\OS, \OT \subset \Rn$ are bounded and open. 
\newcommand{\fv}{T}
Consider the set, $\M$,  of maps which rearrange the measure  $\rho_\OT$ into the measure $\rho_\OS$,
\begin{equation}
\label{jacob}
{\M} = \{ \fv : \OS \mapsto \OT, \,\,\, \rho_\OT(\fv) \, \det(\nabla \fv) = \rho_\OS \}. 
\end{equation}
The OT problem,  in the case of quadratic costs, is given by
\begin{equation} 
\label{mkp2} 
 \inf_{\fv \in \M} \dfrac{1}{2} \int_{ \OS }  \|x-\fv (x) \|^{2} \rho_\OS(x)dx.
\end{equation}
See 
%Figure~\ref{fig:ellipse} 
\autoref{fig:ellipse}
for an illustration of the optimal map between ellipses.

General conditions under which the OT problem (\ref{jacob},~\ref{mkp2}) is well-posed are established in~\cite{MR1100809}.
The unique minimizing map, $\fv$, at which the minimum is reached, is the
gradient of a convex  function %$u:\OS\subset \Rd \to \R$, 
\[
\fv = \grad u, \quad \text{ $u$ convex $:\OS\subset \Rd \to \R$},
\]
which is therefore also unique up to a constant. 
Write %$\nabla u$ for the gradient and 
$D^2 u$ for the Hessian of the function $u$.
Formally substituting $\fv = \grad u$ into~\eqref{jacob} results in the \MA PDE
\bq
\label{MA} \tag{MA}
\det ( D^2 u (x))  = \frac{\rho_\OS(x)}{\, \rho_{\OT} (\nabla u(x))},  \quad \text{ for }  x \in \OS,
\eq
along with the restriction  
\bq\label{convex}\tag{C}
u \text{ is convex}.
\eq
The PDE~\eqref{MA} lacks standard boundary conditions.  However, it is constrained by the 
fact that the gradient map takes  $\OS$ to $\OT$,
\begin{equation} 
\label{BV2}
 \tag{BV2}
\nabla u(\OS) = \OT. 
\end{equation}
The condition~\eqref{BV2} is referred to as the \emph{second boundary value problem for the Monge-Amp\`ere equation} in the literature (see~\cite{MR1454261}).   We will also
 use the term \emph{OT boundary conditions}. \\
\newcommand{\PDE}{(\ref{MA},~\ref{BV2},~\ref{convex})\xspace}

Existing approaches, to solving the OT problem using PDE methods fail to address the condition~\eqref{BV2} for general densities.  
%There are several ways to avoid the general conditions~\eqref{BV2}  but they lead to other problems.  
Instead, they  generally limit their applicability to rectangular  density supports  and use either periodic or inhomogeneous Neumann boundary conditions, provided the mass of the target density does not vanish.  
%This follows from the fact that sides of the square are mapped to the same sides. 

Another way to around this problem is to allow densities to vanish, however this introduces new problems:
Allowing $\rho_\OS$ to vanish causes the the PDE~\eqref{MA} to  become degenerate elliptic.
Allowing $\rho_\OT$ to be discontinuous causes instabilities, since 
 the Lipschitz constant of $\rho_\OT$ appears in gradient descent type of methods.
Smoothing and adding mass to bound either density away from zero changes the optimal map. 
The CFD formulation in \cite{MR1738163} can 
deal with characteristic function densities but its time extended computational domain 
penalizes the computational cost. 

More details and references and given in section 1.5 of \cite{SBVP_Num}.\\

The discrete approximation of the combined problem~\PDE in the generality of 
~\eqref{DensityAss}  is our present topic.

\subsection{Contributions of this work}\label{sec:contribution}

\newcommand{\PDEL}{(\ref{MA},~\ref{OT1},~\ref{convex})\xspace}

In order to enforce  the state constraint (\ref{BV2}) for all possible target sets $\OT$, we use the  signed distance function, $d$, to the boundary of the  domain $\OT$.   
We replace ~\eqref{BV2} by  a Hamilton-Jacobi equation on the boundary:
\bq \label{OT1}\tag{HJ}
\Hd(\grad u(x)) = d(\grad u(x), \OT) = 0, \quad \text{ for } x \in \partial \OS.
\eq
From the convexity of  the domain, $\OT$, the signed distance function, $d$, is also convex.  This implies convexity of $\Hd$.   Using convexity of the solution, 
 $u$, we show that~\PDE is equivalent to~\PDEL.
However, the latter system is more amenable to computations, because it involves a local PDE (a convex Hamilton-Jacobi equation) instead of a global condition on the mapping.  In full generality, any strictly convex Hamiltonian function which vanishes on $\partial \OT$  could be used. This corresponds to the notion of \emph{defining function} introduced by Delano\"e and Urbas in \cite{MR1136351,MR1454261}.
The choice of the signed distance function is made for simplicity.

We recall  it has become common practice to use a distance function to determine a set, as is the case in the level set method.  In {that} case,  a Hamilton-Jacobi equation {is used} to \emph{solve for} the distance function.  What we are doing here is the converse, using the distance function to a {particular} set to \emph{represent} the nonlinear PDE operator.\\

Using this reformulation, we are then able to build approximations and prove convergence of the approximations in the setting of viscosity solutions (which include the regular case).  Our main result, \autoref{thm:convergence} below, is restated here.

\begin{theoremx}[Convergence]
Let $u$ be the unique convex viscosity solution of~(\ref{MA},~\ref{BV2},~\ref{convex}). 
Suppose that $\rho_\OT$ is Lipschitz continuous with convex support, as in the assumption on the densities~\eqref{DensityAss}.
Let $u^{h,d\theta,\da}$ be a solution of the finite difference scheme~\eqref{eq:scheme}.  Then $u^{h,d\theta,\da}$ converges uniformly to $u$ as $h,d\theta,\da\to0$. 
\end{theoremx}

\begin{remark}
{This} theorem does not give a rate of convergence, which is typical of this kind of convergence result, since viscosity solutions can be singular.  However, the filtered scheme which we use is formally second order accurate, so, in the case of smooth solutions, we expect second order accuracy.
\end{remark}

\subsection{Weak solutions of the OT problem}

Any convergent approach {to} the solution of the OT problem in the continuous setting must use an appropriate notion of solution. 
The optimal transportation problem with quadratic cost has a solution under quite general conditions on the measures to be transported. 
It frequently leads to singular solutions of~(\ref{MA},~\ref{BV2},~\ref{convex}), so some notion of weak solutions is needed in the general setting. 
 It is natural to ask under what conditions the system (\ref{MA},~\ref{BV2},~\ref{convex}) is well-posed.

For the convenience of the reader, we briefly discuss regularity and weak solutions, for the purpose of using these notions of solutions for building approximations.  This is a substantial topic, which we can not cover here, we refer to~\cite[Chapter 4]{Villani} and the book~\cite{GutierrezMABook}.

 The existence of classical solutions to~(\ref{MA},~\ref{BV2},~\ref{convex}) requires a smooth, convex target, and smooth strictly positive densities.  The regularity theory for classical solutions can be found in~\cite{MR1136351,MR1454261} and  \cite{MR1426885}. 

Aleksandrov solutions~\cite{GutierrezMABook} are defined in terms of the Hessian measure of the potential function.  This allows for the target measure to be singular with respect to Lebesgue measure.  A typical example is the potential function of a cone $u(x) = \norm{x}$, which is the potential for the mapping between a Dirac measure and a (scaled) characteristic function of the unit ball.  

Pogorelov solutions~\cite{MR1423367}  use specific  target density in the form of a sum of  weighted Dirac measures.  The convex potential function is constructed by lifting affine functions with gradients prescribed by the support of the Dirac measures in the image.  Adjusting functions changes the measure of the support of the affine function so that it matches the target value. 

Convex duality via the Legendre transform is the theoretical idea which links two notions of weak solution provided by Aleksandrov solutions and Pogorelov solutions.  The basic building blocks are polyhedral functions, and their duals under the Legendre transform.    

Brenier's solutions~\cite{MR1100809} are the most general discussed herein.  They allow for the example of tearing a disc, see \autoref{sec:exSplit}.

Viscosity solutions~\cite[page 130]{Villani}, which require continuous density functions, are defined using the comparison principle.  In the case where both measures are absolutely continuous, viscosity solutions coincide with Alexandrov solutions.

For  data in the form of Dirac masses, the constructive  method of \emph{Pogorelov solutions} is natural for the OT problem.  In this setting,~\eqref{BV2} is naturally satisfied, and the gradient mapping is approximated instead of the potential, which reduced the discretization error of the mapping. 
This approach formed the  basis of  several  early numerical methods:~\cite{olikerprussner88, MR1089128,McCannGangbo} and more recently  \cite{bosc, merigot}.   Applying this approach  to more  general 
density functions requires quantization which introduces errors which are difficult to estimate.  Also
the best algorithm (assuming the initial density is also in the form of a sum of Dirac masses), 
the \emph{auction} algorithm \cite{MR1195629},  scales as $\bO(N^2logN)$ in the number of Dirac masses.

Viscosity solutions, though less general than Aleksandrov solutions,  are well-suited to finite difference methods. They allow for optimal  maps with  
discontinuous gradients to be computed. Densities are easily discretized on grids.  Also, there is a robust convergence theory~\cite{BSnum} which after 
the reformulation   of ~\PDE into~\PDEL  can be applied. Finally, in previous work we have developed fast and robust  numerical methods for computing viscosity solutions of the  \MA equation with standard boundary conditions (see \autoref{sec:prevwork}).   Moreover, the numerical results in 
\cite{SBVP_Num}  indicate that this approach reduces to the computational cost of a linear elliptic solver which is 
 log-linear with respect to  the number of grid points. 
 
%Rigorous implementations using other notions of solution are also possible, see the discussion of related work which follows.

\subsection{Relation to our previous work}\label{sec:prevwork}

A lot of attention has been devoted to the numerical solution of the {Monge-Amp\`ere}
equation since the pioneering work of \cite{DGnum2006}.  

This article builds on a series of papers which have developed convergent and robust solution methods for the {Monge-Amp\`ere} equation.  The definition of elliptic difference schemes was presented in~\cite{ObermanDiffSchemes}, which laid the foundation for the schemes that followed.  See the references therein for a comprehensive reference list on numerical methods 
for Monge-Amp\`ere equation.

A first convergent scheme for the \MA equation was built in~\cite{ObermanEigenvalues}; it was restricted to two dimensions and to a slow iterative solver.   Implicit solution methods were first developed in~\cite{BenamouFroeseObermanMA}, where it was demonstrated that the use of non-convergent schemes led to slow solvers for singular solutions.   
In~\cite{ObermanFroeseMATheory} a new discretization was presented, which generalized to three and higher dimensions;  this {also} led to a fast Newton solver.  

The wide stencil schemes used in the convergent discretizations introduced a new parameter, the directional resolution, which led to decreased accuracy.  In an attempt to improve accuracy on less singular solutions, a hybrid discretization was built in~\cite{ObermanFroeseFast}. This discretization combined the advantages of accuracy in smooth regions, and robustness (convergence and stability) near singularities.   However, this was accomplished at the expense of a convergence proof.   In addition, it required \emph{a priori} knowledge of singularities, which is not available in the OT setting.    
  
In~\cite{ObermanFroeseFiltered}{,} the convergence theory of Barles and Souganidis was extended to  filtered (nearly monotone) schemes.  The filtered schemes replaced the hybrid schemes as a method to obtain more accuracy, without losing the convergence proof. On less singular solutions,  the filtered schemes can be used with a compact stencil eliminated the complication and slightly increased cost of the wide stencil.

The  numerical resolution of  ~\PDE was first  addressed in~\cite{FroeseTransport}. 
The method consisted of iteratively solving~\eqref{MA} with Neumann boundary conditions, and projecting the resulting set onto the target set $\OT$.  
The new projection is then used to derive new Neumann boundary conditions.  
This method required several iterations to reach a solution satisfying the state constraint (\ref{BV2}). No convergence proof was available. 
We show in \cite{SBVP_Num}  that this method can be interpreted as a particular solution method for~(\ref{MA},~\ref{BV2},~\ref{convex}).

\section{Representation and approximation of the boundary conditions}

%Rigorously addressing the boundary conditions~\eqref{BV2} in the context of viscosity solutions for the PDE~\eqref{MA} requires the foundation of a robust, convergent discretization and new theoretical and practical ideas.  This is why it has taken longer than the development of convergent methods for the Dirichlet problem. 

 In this section we describe our approach to approximating the boundary condition ~\eqref{OT1}  based on the signed distance function to the boundary of the target set~$\OT$.
%We are able to treat the boundary condition using a monotone finite difference scheme, which is consistent with the treatment of the PDE~\eqref{MA}.

%The resulting  boundary conditions are implicit, in contrast to, for example, Dirichlet or Neumann boundary conditions, which have been previously implemented for~\eqref{MA}, but they are local, in contrast to the original formulation.

\subsection{Representation and properties of the distance function}\label{sec:disfun}

%
%Our formulation of the boundary conditions~\eqref{BV2} then becomes
%\bq\label{OT2}\tag{HJ}
%\Hd(\grad u(x)) = 0, \quad \text{ for } x \in \partial \OS
%\eq
For a bounded, convex set $\OT$, the signed distance function, 
%$d(y, \OT)$, 
%defined in~\eqref{signeddist},
\bq\label{signeddist}
d(y, \OT) =
\begin{cases}
+ \dist(y, \partial \OT), \quad \text{ $y$ outside of $\OT$,}
 \\
- \dist(y, \partial \OT), \quad \text{ $y$ inside $\OT$.}
 \end{cases}
\eq
 is convex.
%\begin{remark}
%The unsigned distance function can be represented as $\dist(y,\partial \OT) = \inf_{y_0 \in \partial \OT} \norm{ y-y_0}$, but since it is not convex, it is not suited to our purposes.  
%\end{remark}
 
The function 
\[
\Hd(y) = d(y,\OT),
\]
 which is now interpreted as a convex Hamiltonian, can be written in terms of the supporting hyperplanes to the convex set,
\bq\label{HPlanes}
\Hd(y) = \sup_{y_0 \in \partial \OT}  \left \{n(y_0) \cdot ( y - y_0) \right \}, 
\eq
where $n(y_0)$ is the outward normal to $\partial \OT$ at $y_0$; see~\autoref{fig:shape}. 
 Equivalently, since the image of the normals to $\partial \OT$ is the unit sphere, we can write
\bq\label{Hnu}
\Hd(y) = \sup_{\norm{n} = 1}  \{n \cdot (y - y(n) )  \}, %=  \sup_{\norm{n} = 1}  \{n \cdot y - H^*(n)  \}
\eq
where $y(n)$ is the point in $\partial \OT$ with normal $n$. 
% Here $\Hd^*$ is the Legendre Transform of $\Hd$.   
The representations~\eqref{HPlanes} and~\eqref{Hnu} are equivalent via duality: in the first case the convex set is represented by normals, in the second by points.
We found the second representation to be useful, since sets are naturally approximated by a characteristic functions: this was  the representation used in~\cite{SBVP_Num}.

The statements above follow from the \emph{Supporting Hyperplane Theorem}~\cite[Section 2.5]{BoydBook}, which says that if $y_0 \in \partial \OT$, for a convex set $\OT$, 
then $y_0$ has a (possibly non-unique) supporting hyperplane, 
\[
P  = \{ A(y) = 0  \mid  A(y) \equiv n\cdot (y -y_0) \},
\]
 where $A(y) \le 0$, for $y\in \OT$.  Without loss of generality, $\norm{n}=1$, and we can define $n$ to be (an) outward normal to $\OT$ at $y_0$.  Then, we can define 
\[
H^*(n) = n\cdot y(n) = \sup_{y\in \partial \OT} n\cdot y,
\]
where the equality follows from the supporting hyperplane result.  We have proven the following lemma.

\begin{figure}
\begin{center}
\includegraphics[width=0.5\textwidth]{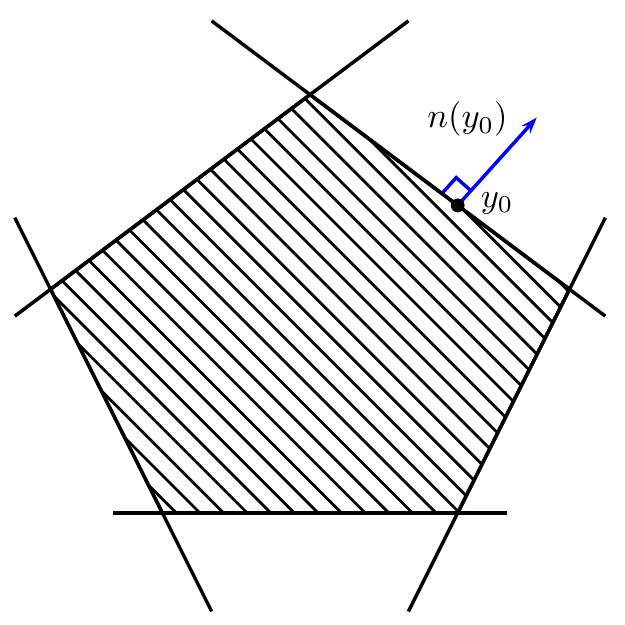}
\end{center}
\caption{Polyhedral target set. }
\label{fig:shape}
\end{figure}

\begin{lemma}\label{lem:distrep}
Let $\Hd(y) = d(y,\OT)$ be the signed distance function to the smooth, convex, bounded set, $\OT$.  Then 
%\bq\label{H_LF2}
\[
H(y) = \sup\limits_{\norm{n} = 1}\{y\cdot n - H^*(n)\}
\]
where 
\bq\label{LF3} 
H^*(n) = \sup\limits_{y_0\in\partial\OT}\{y_0\cdot n\}. 
\eq
\end{lemma}

\begin{remark}[Approximation of $\Hd^*$]
Later we will make an approximation by taking the supremum over 
a finite subset of the admissible directions.  These direction vectors are typically given by a uniform discretization of $[0,2\pi]$, with angle discretization parameter $d\alpha$.  We require only that $d\alpha\to0$ for convergence.
\end{remark}

\subsection{Obliqueness}

We recall here a fundamental property of maps characterized as the gradient of a convex potential.
In \cite{MR1136351,MR1454261}, this {\em obliqueness} result is 
used to prove existence of classical solutions to  (\ref{MA}-\ref{BV2}). 
%More precisely, it corresponds to  the property of the linearization of  (\ref{BV2}). 

This condition, which leads to Lemma~\ref{lem:directions}, will allow us to build an explicit monotone upwind  discretization of~\eqref{HPlanes}, using points inside the domain.

\begin{figure}[hbt]
\begin{center}
\centering
\includegraphics[width=\textwidth]{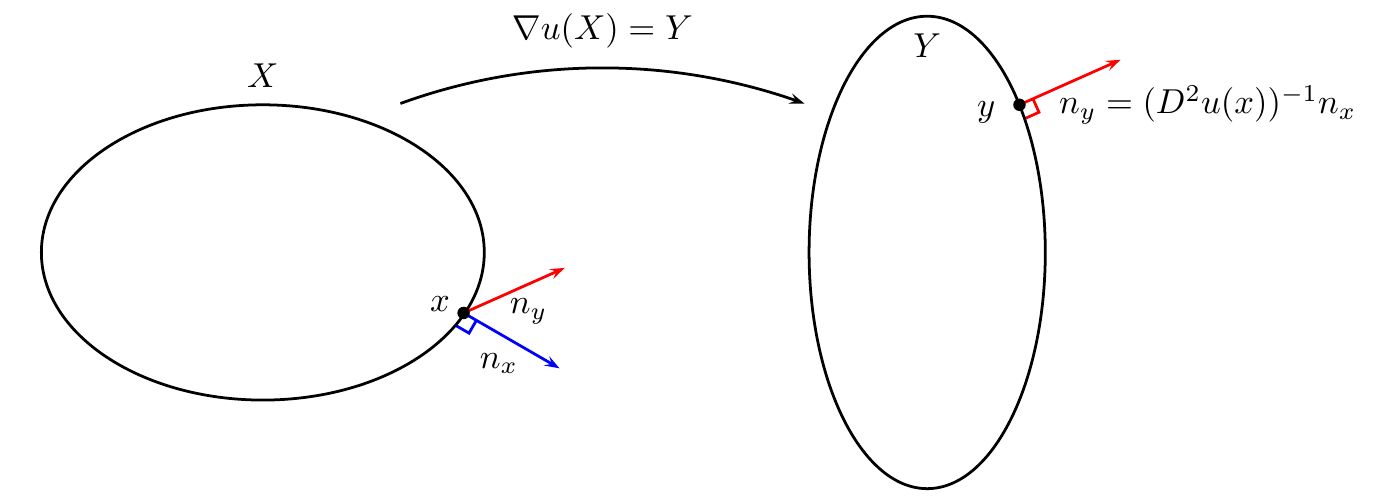}
\end{center}
\caption{Illustration of the mapping
$y = \grad u(x) $ and the  normal vectors. 
} 
\label{fig:boundarymapping}
\end{figure}

\begin{lemma}\label{lem:oblique}
Suppose $\OS$ is  a convex domain, and $\OT = \grad u (\OS)$ is the image of {$\OS$} under the mapping $\grad u$, where $u$ is a convex twice continuously differentiable function.  Then the normal vectors 
%$n_x$ and $n_y$ 
make an acute angle,
\bq\label{oblique}
n_x \cdot n_y \ge 0.
\eq
See \autoref{fig:boundarymapping}. 
\end{lemma}

\begin{proof}
Let $\grad u(x) = y\in \partial \OT$ and let $\Hd(y) = \dist(y,\partial Y)$ be the signed distance function to $Y$, so that 
\[
\OT = \{ y \mid \Hd(y) = 0 \},
\]
and 
\[
\OS = \{ x \mid \Hd(\grad u (x)) = 0 \}.
\]
Then $n_y = \grad H (y)$
and, by the chain rule for differentiation,
\[
n_x = c \grad D^2 u(x) \grad \Hd(\grad u(x)) = c D^2u(x) \grad \Hd(y)
\]
for a normalization constant  $c$. 
Thus
\[
n_x \cdot n_y = c (\grad \Hd(y))^T  D^2u(x) \grad \Hd(y) \ge 0,
\]
since convexity of $u$ means $D^2u$ is positive definite.
\end{proof}

We are now able to give the characterization of (\ref{OT1}) which will be used for the monotone discretization.

Note, it is enough to check this condition on $C^2$ functions, because in the definition of viscosity solutions, the boundary conditions are tested in the weak sense, which means checking the appropriate inequality when touching from above or below by a $C^2$ function.
In particular, we wil use oblique boundary conditions and check equations (4) and (5) in \cite{barles1999nonlinear}. 
\begin{lemma}\label{lem:directions}
Let $u\in C^2(\OS)$ be any convex function that satisfies~\eqref{BV2}.  For any $x\in\partial\OS$ with unit outward normal $n_x$, the supremum in~\eqref{OT1} can be restricted to vectors making an acute angle with $n_x$:
\bq\label{HJoblique}
\Hd(\grad u(x)) = \sup\limits_{\norm{n}=1}\{ \grad u(x) \cdot n - \Hd^*(n) \mid n\cdot n_x > 0\} = 0.
\eq
\end{lemma}

\begin{proof}
The supremum above will be attained for a value of $n = n^*$, which will be identical to the unit outward normal to the target at the point $\nabla u(x)$.

From Lemma~\ref{lem:oblique}, we know that $n^*\cdot n_x = n_y\cdot n_x \geq 0$.  Consequently, it is only necessary to check values of $n$ that make an acute angle with the boundary of the domain.
\end{proof}

\subsection{Monotone discretization of $\Hd$}\label{sec:discH}
In this section we explain how to build a monotone discretization of $\Hd$ using points at the boundary and on the inside of the domain $\OS$.

The expression~\eqref{HJoblique}, which comes from writing the convex set $\OT$ in terms of its tangent hyperplanes, leads to a natural convergent finite difference discretization.  This expression can be interpreted as a Hamilton-Jacobi-Bellman equation arising from an optimal control problem.   

We recall (see Definition~\ref{def:Ell}) that an elliptic (monotone) discretization of $\Hd(\nabla u)$ takes a particular form.  In the case at hand, the discretization is given in the following form.  At the point $x_i$ the discretization is  a nondecreasing function of the differences between $u(x_i)$ and $u(x_j)$  where $x_{j}$ are the neighbours of $x_i$.

If we let $\Gamma_{x_i} = \{n \mid n\cdot n_{x_i}>0, \norm{n}=1\}$, 
a simple way of writing an upwind discretization is to approximate the signed distance function by
\[ 
\Hd(\nabla u(x_i)) = \sup\limits_{n\in\Gamma_{x_i}}\{\nabla u(x_i)\cdot n - \Hd^*(n)\}
\approx \sup\limits_{n\in\Gamma_i}\left\{\frac{u(x_i) - u(x_i - n h)}{h} - \Hd^*(n)\right\} 
\]
where $h$ is a small discretization parameter.  The discretization of the linear advection above can be directly implemented in the 
wide stencil framework as in~\cite{ObermanEigenvalues} (see also section 3.3 below).
As long as the domain is uniformly convex and $h$ is sufficiently small, obliqueness ensures that the point $x_i - n h$ is inside the domain.  This guarantees that the monotonicity for all $n \in\Gamma_i$.  Taking  the supremum of monotone terms results in a monotone expression. 

Alternately, we can use standard compact upwind finite differences on a grid  for the linear advection equation.
Assume for simplicity that the boundary is linear and  the local coordinate system such 
 that at the boundary point $x_{i,j}$, the normal is $n_{x_{i,j}}= (-1,0)$.  
Along this side, the set of admissible directions $\Gamma_{x_{i,j}}$ will be given by
\[ \{n = (n_1,n_2) \mid n_1 < 0, \norm{n} = 1\}. \]
Then, letting $dx$ denote the spatial resolution of the grid, we can approximate the advection terms in (\ref{HJoblique}) by
\begin{multline*} \nabla u(x_{i,j})\cdot n \approx n_1 \frac{u(x_{i+1,j})-u(x_{i,j})}{dx} \\
+ \max\{n_2,0\}\frac{u(x_{i,j}) - u(x_{i,j-1} )}{dx} + \min\{n_2,0\}\frac{u(x_{i,j+1})-u(x_{i,j})}{dx}. 
\end{multline*}
Due to the obliqueness property again, upwinding leads to a scheme that 
 relies on values inside the square and, because $n_1 < 0$, it is monotone.
Taking the supremum of these monotone schemes over all admissible directions, we again 
preserve monotonicity of the scheme. 
%%%%%%%%%%

We follow this last approach in \cite{SBVP_Num} where we always use an initial density,  $\rho_X$,  
whose support is embedded in a rectangular domain and padded by zeros. The robustness of our 
\MA solver to degenerate solution allows for this simplification.

\section{Convergence} 
We begin with a review of background material that will be needed to construct and prove the convergence of our scheme for solving the second boundary value problem for the \MA equation.

The viscosity approximation theory developed by Barles and Souganidis~\cite{BSnum}  provides criteria for the convergence of approximation schemes:  schemes that are consistent, monotone, and stable converge to the unique viscosity solution of a degenerate elliptic equation.    This framework was extended in~\cite{ObermanFroeseFiltered} to nearly monotone schemes, which are more accurate.

In order to prove convergence, we need to know that there exist unique solutions of the system \PDE.  The boundary condtions are nonstandard, so the uniqueness of the system is not coverged by the standard literature.  However, uniqueness for these types of boundary conditions is proved in \cite{barles1999nonlinear}.

\subsection{Testing the Barles conditions for obliqueness}

\blue{Need to explain how (H1) on page 3 is satisfied by our operator} 

The obliqueness condition (H1) in Barles, written in our notation is 
\[
H(x,p + \lambda n_x) - H(x,p) \ge C \lambda, \quad \text{ for } x \in \partial \OS
\]
This follows from \eqref{HJoblique} in fact, the constant is $C=1$.  (details supremum of affine functions increasing in $n_x$ with slope 1. \dots ) 

The second condition (H2) is a standard growth condition on $H$ which certainly follows from the fact that $H$ is a 1-Lipschitz function in $p$ and continuous in $x$.

This allows us to conlclude comparison for viscosity solutions. 

\subsection{Writing the combined problem as a single operator}

It is convenient to write the combined PDE and boundary conditions \PDE as a single (discontinuous) operator on $\bar{\OS}$
\bq
\label{FOP} 
F(x,u(x) , \nabla u(x), \nabla^2 u(x)) =  0 , \,\, x \in \bar{\OS}  
\eq
where $F$ depends on $\rho_{\OS},\rho_{\OT}$ and $\Hd$: 
 \bq
\label{FOPdef} 
%\left\{ \begin{array}{ll}
F(x,u,p, M)  = 
\begin{cases}
 \det (M) - \rho_{\OS}(x)/\rho_\OT(p),  &  \,  x  \in \OS  % , \,  p \in \R^d , \, M \in   \R^{d\times d} 
\\ 
 \Hd(p), & \, x \in \partial \OS,%p \in \R^d 
\end{cases} 
\eq
along with the convexity condition~\eqref{convex}.

\begin{example}
The Dirichlet problem is simply given by replacing $\Hd(p)$ with $u-g$ for $x \in \partial \OS$.
\end{example}

\subsection{Convergence theory}
In this subsection we review the definitions and the general theory for convergence schemes, it is applied to the equation~\eqref{FOPdef} in a subsequent section.

\begin{definition}[Consistent]\label{def:consistent}
The scheme~$\Fe$ is \emph{consistent} with the equation~$F = 0$ if for any smooth function $\phi$ and $x\in\bar{\Omega}$,
\[ \limsup_{\e\to0,y\to x,\xi\to0} \Fe(y,\phi(y)+\xi,\phi(\cdot)+\xi) \leq F^*(x,\phi(x),\nabla\phi(x),D^2\phi(x)), 
\]
\[ \liminf_{\e\to0,y\to x,\xi\to0} \Fe(y,\phi(y)+\xi,\phi(\cdot)+\xi) \geq F_*(x,\phi(x),\nabla\phi(x),D^2\phi(x)). \]
\end{definition}

\begin{definition}[Elliptic]\label{def:Ell}
The scheme~$\Fe$ is \emph{elliptic} if it can be written
\[
\Fe[v] = \Fe(x, v(x), v(x) - v(\cdot)), 
\]
where $\Fe$ is nondecreasing in its second and third arguments, i.\ e.\
\bq\label{felliptic}
s \le t,~ u(\cdot) \le v(\cdot) \implies  \Fe(x,s,u(\cdot)) \le  \Fe(x,t,v(\cdot))
\eq
\end{definition}

\begin{definition}[Nearly Elliptic]\label{def:nearEll}
The scheme~$\Fe$ is \emph{nearly elliptic} if it can be written as
\bq\label{nearEll}
\Fe[v] = \Fm[u] + \Fp[u]
\eq
where $\Fm$ is a monotone (elliptic) scheme and $\Fp$ is a perturbation, which satisfies
\[ \lim_{\eps\to0}\|\Fp\| = 0. \]
\end{definition}

Using these definitions, we now recall the main convergence theorem from~\cite{ObermanFroeseFiltered}, which is an extension of the convergence theory of~\cite{BSnum}.  This almost monotone scheme is related to the work of Abgrall~\cite{Abgrall} which is a ``blending'' of a monotone and a higher-order scheme.  

\begin{theorem}[Convergence of Nearly Monotone Schemes~\cite{ObermanFroeseFiltered}]\label{thm:schemesConverge}
Let $u$ be the unique viscosity solution of the PDE~\eqref{FOP} and let $u^\eps$ be a stable solution of the consistent, nearly elliptic approximation scheme~\eqref{nearEll}.  Then 
\[ u^\e \to u, \quad \text{ locally uniformly,  as } \e \to 0. \]
Moreover, if the non-monotone perturbation $\Fp$ is continuous, $u^\eps$ exists and is stable.
\end{theorem}
The filtered approximation schemes combine a monotone (elliptic) scheme~$\Fm$ with a more accurate, non-monotone scheme~$\Fa$.  These schemes, which are a particular type of nearly monotone scheme (thus they are convergent by the theorem above), have the form
\bq\label{eq:filtered} \Fe[u] = \Fm[u] + r(\epsilon) S\left (\frac{\Fa^\eps[u]-\Fm^\eps[u]}{r(\epsilon)}\right) \eq
where $r(\epsilon)\to0$ as $\epsilon\to0$.
Here the function $S$, which is called a filter, can be defined, for example, by
\bq\label{eq:filter}
S(x) = \begin{cases}
x & \norm{x} \leq 1 \\
0 & \norm{x} \ge 2\\
-x+ 2  & 1\le x \le 2 \\
-x-2  & -2\le x\le -1.
\end{cases} 
\eq
See \autoref{fig:filter}.

\begin{figure}[hbt]
% trim l b r t
\includegraphics[trim=1.6in 1.9in 2in 2.5in, clip=true] 
{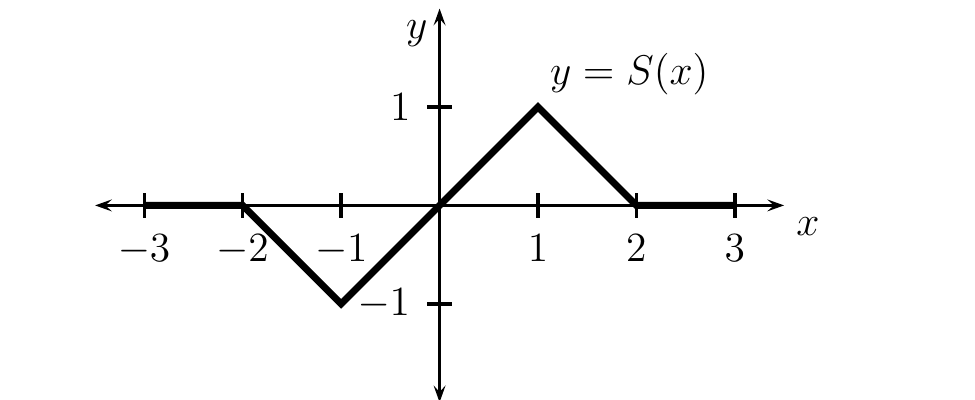}
\caption{
Filter function
}\label{fig:filter}
\end{figure}

\subsection{Discretization of \MA equation}\label{sec:discMA}
 The equation we want to solve is
\[ \det(D^2u(x)) = \rho_\OS(x)/\rho_\OT(\nabla u(x)) + \langle u\rangle. \]
%However, to reduce the cost of computations, we will replace the mean $\left<u\right>$ with the value $u(x_0)$ for some point $x_0\in\OS$, which should coincide with a grid point.  The solution to the two problems is the same up to an additive constant.

We first describe the elliptic (monotone) scheme for the \MA operator, which underlies the filtered scheme.
This scheme was developed in~\cite{ObermanFroeseMATheory,FroeseTransport}.  

\begin{remark}
It is somewhat complicated to describe the discretization.  First, we combine a discretization of the \MA operator with the convexity constraint.  In addition, regularization terms are added to make the operator differentiable when Newton's method is applied.  Further, the dependence on the gradient in~\eqref{MA} means that small changes in the values of $u$ can lead to large changes in the equation.  This causes instabilities when nonconvergent schemes are used.  We also modify the discretized \MA operator to compensate for the dependence on the gradient.  We describe the modifications step by step.
\end{remark}

To begin, we use Hadamard's inequality to represent the determinant of a positive definite matrix, 
$
\det(M) \le \Pi_{i} m_{ii},
$
with equality when $M$ is diagonal.   Then, we can write
\[
\det(M) = \min 
\left \{ \Pi_{i} (O^T M O)_{ii}
\mid O^TO = I
\right \}
\]
where $O$ is an orthogonal matrix.  This last inequality, applied to the Hessian of a convex function, corresponds to taking products of second derivatives of the function along orthogonal directions,  
\[\det(D^2u) \equiv \min\limits_{\{\nu_1\ldots\nu_d\}\in V} 
\prod\limits_{j=1}^{d} 
u_{\nu_j\nu_j},\]
where $V$ is the set of all orthonormal bases for $\R^d$.

In the special case where the source density $\rho_\OS$ vanishes, the \MA operator reduces to the convex envelope operator~\cite{ObermanConvexEnvelope, ObermanCED} which corresponds to  directional convexity in each direction.  The convex envelope operator  is approximated by enforcing directional convexity in grid directions~\cite{ObermanEigenvalues,ObermanCENumerics}.
 
In the current setting, the convexity constraint is enforced by additionally replacing the directional  derivatives with their positive part.  In addition, to prevent non-convex solutions when the right-hand side vanishes, we will {also} subtract the negative parts of these second derivatives,  
\[{\det}^+(D^2u) \equiv \min\limits_{\{\nu_1\ldots\nu_d\}\in V} \left\{
\prod\limits_{j=1}^{d} 
\max\{u_{\nu_j\nu_j},0\} + \sum\limits_{j=1}^d\min\{u_{\nu_j\nu_j},0\}\right\},
\]
which is valid when $u$ is convex.
These modifications ensure that a non-convex function cannot solve our \MA equation (with non-negative right-hand side).

We discretize the operator above in two ways.  First, we make the conventional step of  replacing derivatives by finite differences.  Second, instead of considering all orientations, we replace $V$ with a finite subset which uses only  a finite number of vectors $\nu$ that lie on the grid and have a fixed maximum length.  The second discretization is called the the directional discretization, and we quantify it using $d\theta$, the angular resolution $d\theta$ of our stencil (see \autoref{fig:stencil}).  In this figure, values on the boundary are used to maintain the directional resolution $d\theta$ (at the expense of lower order accuracy in space because the distances from the reference point are not equal).  Another option is to narrower stencils as the boundary is approached, which leads to lower angular resolution, but better spatial resolution.  We denote the resulting set of orthogonal vectors by $\G$.

\begin{figure}[htdp]
	\centering
	\subfigure[]{\includegraphics[height=.4\textwidth]{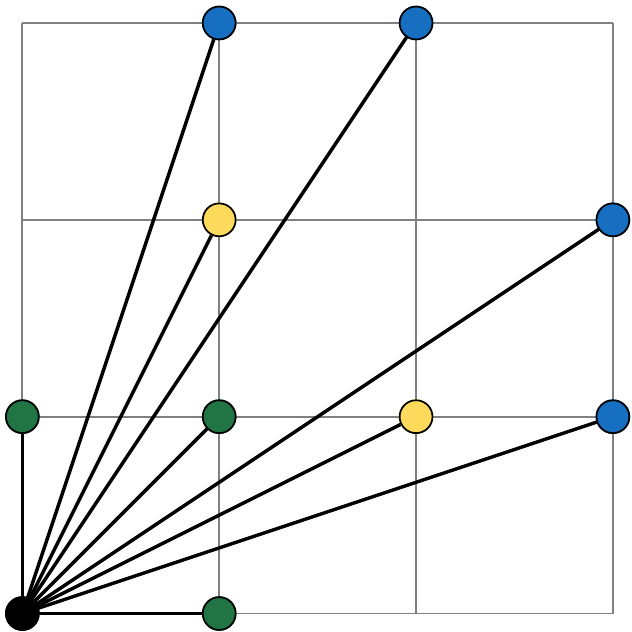}}
  \subfigure[]{\includegraphics[height=.4\textwidth]{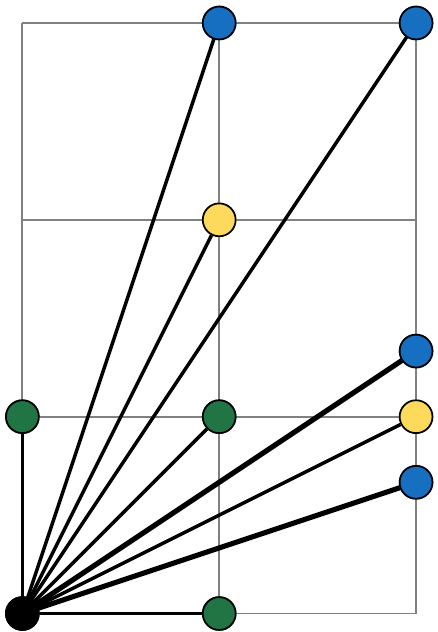}}
  \subfigure[]{\includegraphics[height=.4\textwidth]{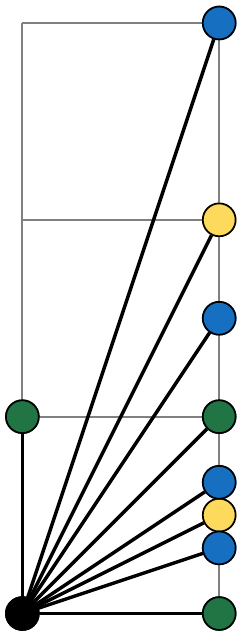}}
  	\caption{Neighboring grid points used for width one  (green), two (yellow), and three (blue) stencils.  The illustration shows the neighbors in the first quadrant.  The modification near the boundary is illustrated in the second and third figures.
}
  	\label{fig:stencil}
\end{figure}

Each of the directional derivatives in the \MA operator is  discretized using centered differences:
\[ \Dt_{\nu\nu}u_i = \frac{1}{\norm{\nu}^2h^2}\left(u(x_i + \nu h) + u(x_i - \nu h) - 2u(x_i)\right). \]

In order to handle non-constant densities, we also need to discretize the gradient.
%To discretize the gradient, which is needed for non-constant target densities, we make use of derivatives in these same directions.  This allows us to use centered differences without sacrificing the monotonicity of the equation as a whole:
%\[ \Dt_{\nu}u_i = \frac{1}{2\norm{\nu}h}\left(u(x_i+\nu h)-u(x_i-\nu h)\right). \]
This is accomplished using the Lax-Friedrichs scheme, which uses centered differences augmented by  a small multiple of the Laplacian to ensure monotonicity. 
\[ -\rho_\OS/\rho_\OT(\nabla u) \approx -\rho_\OS/\rho_\OT\left(\Dt_{x_1}u,\ldots,\Dt_{x_d}u\right) + \delta \sum\limits_{j=1}^d \Dt_{x_jx_j}u. \]
The centered difference discretization of the first derivatives is given by
\[ \Dt_{\nu}u_i = \frac{1}{2h}\left(u(x_i + \nu h) - u(x_i - \nu h)\right). \]
To preserve monotonicity, we require the parameter $\delta$ to satisfy 
\bq\label{deltaK}
\delta > K h,
\eq
where $K$ is the Lipschitz constant (with respect to $y$) of $\rho_\OS(x)/\rho_\OT(y)$.

%\begin{remark}\label{rem:disc}
%In practice, we do not add a multiple of the laplacian, but instead absorb the parameter $\delta$ into the second-derivative operators that are already present in the equation.  The gradient can then be discretized using centered differences along rotated coordinate frames, which correspond to the directions that are active in the \MA operator.  This approach, which is described in~\cite{FroeseTransport}, leads to sparser systems and improves the consistency error of the monotone scheme.  These differences do not affect the convergence proof or the formal consistency error of the more accurate filtered scheme.
%\end{remark}

Then an elliptic discretization of the \MA equation is
\bq\label{eq:MAmonDisc}
MA_M^{h,d\theta,\delta}[u] 
  = \min\limits_{(\nu_1,\ldots,\nu_d)\in \G}G_{(\nu_1,\ldots,\nu_d)}^{h,d\theta,\delta}[u]
\eq
where each of the $G_{(\nu_1,\ldots,\nu_d)}^{h,d\theta,\delta}[u]$ is defined as
%\bq\label{eq:monterm}
%\begin{split}
%G_{(\nu_1,\ldots,\nu_d)}^{h,d\theta,\delta}[u] = &\prod\limits_{j=1}^d\max\{\Dt_{\nu_j\nu_j}u,0\} + 
%\sum\limits_{j=1}^d\min\{\Dt_{\nu_j\nu_j}u,0\} - \\
%&\rho_\OS(x) / \rho_\OT\left(\sum\limits_{j=1}^d\frac{\nu_j\cdot \ev_1}{\abs{\nu_j}}\Dt_{\nu_j}u,\ldots,\sum\limits_{j=1}^d\frac{\nu_j\cdot \ev_d}{\abs{\nu_j}}\Dt_{\nu_j}u\right) - u(x_0).
%\end{split}
%\eq
\bq\label{eq:monterm}
\begin{split}
G_{(\nu_1,\ldots,\nu_d)}^{h,d\theta,\delta}[u] = &\prod\limits_{j=1}^d\max\{\Dt_{\nu_j\nu_j}u,0\} + 
\sum\limits_{j=1}^d\min\{\Dt_{\nu_j\nu_j}u,0\}\\
&-\rho_\OS(x) / \rho_\OT\left(\Dt_{x_1}u,\ldots,\Dt_{x_d}u\right) + \delta\sum\limits_{j=1}^d\Dt_{x_jx_j}u - u(x_0).
\end{split}
\eq

This monotone scheme forms the basis of the filtered scheme~\eqref{eq:filtered}.  For improved accuracy on smooth solutions, we combine {the monotone scheme} with the accurate scheme~$\Fa$, which is simply a standard centered difference discretization of the (two-dimensional) equation
\[ u_{x_1x_1}u_{x_2x_2}-u_{x_1x_2}^2 - \rho_\OS(x)/\rho_\OT(u_{x_1},u_{x_2}) - u(x_0). \]
We denote the resulting discretization by~$MA_S[u]$.
A similar discretization is easily constructed in higher dimensions by using standard finite difference discretizations for the other terms.  Additional details can be found in~\cite{ObermanFroeseFast}.

\subsection{Convergence to the viscosity solution}
We combine the almost monotone schemes for~(\ref{MA}, \ref{convex}) with the upwind, monotone scheme for~\eqref{OT1} into one equation, which we show converges to the unique convex viscosity solution of the system~\PDE.  
The combined scheme is given as
\bq\label{eq:scheme}
F^{h,d\theta,d\alpha}[u]_i = 
\begin{cases}
F_F^{h,d\theta}[u]_i & x_i \in \OS\\
\Hd^{h,d\alpha}[u]_i & x_i \in\partial\OS.
\end{cases}
\eq
In this definition, $F_F$ is the filtered scheme for the \MA equation~\eqref{eq:filtered}, which relies on the discretizations described in \autoref{sec:discMA}, and $\Hd$ is the upwind discretization of the boundary condition described in \autoref{sec:discH}.

\begin{theorem}[Convergence]\label{thm:convergence}
Let $u$ be the unique convex viscosity solution of~\PDE. 
Suppose that $\rho_\OT$ is Lipschitz continuous with convex support, as in the assumption on the densities~\eqref{DensityAss}.
Let $u^{h,d\theta,\da}$ be a solution of the finite difference scheme~\eqref{eq:scheme}.  Then $u^{h,d\theta,\da}$ converges uniformly to $u$ as $h,d\theta,\da\to0$. 
\end{theorem}

\begin{proof}
We can replace \PDE by \PDEL because the viscosity solution $\grad u$ maps $\OS$ to $\OT$, in particular, $\grad u$  maps $\partial \OS$ to $\partial \OT$.  So $H(\grad u(x)) = 0$ for $x \in \partial \OS$.

From \autoref{thm:schemesConverge}, we need only verify that the scheme is consistent and nearly elliptic (nearly monotone).

Consistency and near monotonicity of the scheme for the \MA equation have been established in~\cite{ObermanFroeseMATheory,FroeseTransport}.

We recall the form of the boundary condition in~\eqref{HJoblique},\[
H(\grad u(x)) = \sup\limits_{\norm{n}=1}\{ \grad u(x) \cdot n - H^*(n) \mid n\cdot n_x > 0\}.
\]
This is discretized using forward or backward differences for the gradient, which are consistent as $h\to0$.
The supremum is further approximated by restricting to a finite subset of directions, with angular resolution $\da$.  Since the Legendre-Fenchel transform $\Hd^*(n)$ is continuous, and since these admissible directions are approximated with an accuracy on the order of $\da$, this approximation is consistent as $\da\to0$.

By exploiting the obliqueness property (Lemma~\ref{lem:directions}), we were able to construct an upwind discretization of the boundary condition, which is monotone by construction.
\end{proof}

 \section{Computational Results} 
We now provide some brief computational results to demonstrate that the approximation scheme described in this paper can be used in practice to numerically solve the optimal transportation problem.  
Full details of the numerical method and its implementation are given in a companion paper~\cite{SBVP_Num}.

For the first two examples, which admit exact analytical solutions, we provide the maximum norm of the distance between the mappings obtained from the exact and computed solutions ($u_{ex}$ and $u_{comp}$ respectively):
\[ \max\limits_{x\in\OS}\|\nabla u_{ex}(x)-\nabla u_{comp}(x)\|_2. \]

In the tables, the number of grid points used to approximate each dimension of the domain~$\OS$ and target~$\OT$ are proportional to $N_X$ and $N_Y$ respectively.
 
\begin{remark}[Reading mappings from the figures (\ref{fig:ellipse}-\ref{fig:split}) ]
The source density is embeded in a squared domain discretized by a cartesian grid and padded with zeros. We represent on the left the restriction of the grid to the support of the source and on the right , its image by the optimal map. 
 The optimal mapping can be interpreted from the figures by noting, first, that the centre of masses are mapped to each other.  Next, moving along a grid line in the source, the corresponding point in the target can be found by using monotonicity of the map: the corresponding grid point is in the same direction.  
\end{remark}

\subsection{Mapping an ellipse to an ellipse}\label{sec:exEllipse}
We consider the problem of mapping an ellipse onto an ellipse.  
To describe the ellipses, we let $M_x,M_y$ be symmetric positive definite matrices and let $B_1$ be the unit ball in $\R^d$.
Now we take $X = M_xB_1$, $Y = M_yB_2$ to be ellipses with constant densities $f$, $g$ in each ellipse.

In $\R^2$, the optimal map, which is linear, can be obtained explicitly~\cite{MOEllipse}.  It is given by 
\[ \nabla u(x) = M_yR_\theta M_x^{-1}x\]
where $R_\theta$ is the rotation matrix
\[ R_\theta = \left(\begin{array}{cc} \cos(\theta) & -\sin(\theta)\\ \sin(\theta) & \cos(\theta)\end{array}\right), \]
the angle $\theta$ is given by
\[ \tan(\theta) = \trace(M_x^{-1}M_y^{-1}J)/\trace(M_x^{-1}M_y^{-1}), \]
and the matrix $J$ is equal to
\[ J = R_{\pi/2} = \left(\begin{array}{cc} 0 & -1\\ 1 & 0\end{array}\right). \]

We use the particular example
\[ M_x = \left(\begin{array}{cc}0.8 & 0\\0 & 0.4 \end{array}\right) , 
\quad M_y = \left(\begin{array}{cc} 0.6 & 0.2\\0.2 & 0.8\end{array}\right),\]
which is pictured in \autoref{fig:ellipse}.  

Results are presented in Table~\ref{table:ellipse}, which demonstrates first order accuracy in both $N_X$ and $N_Y$.

\begin{figure}[htdp]
	\centering
	\subfigure[]{\includegraphics[width=.495\textwidth]{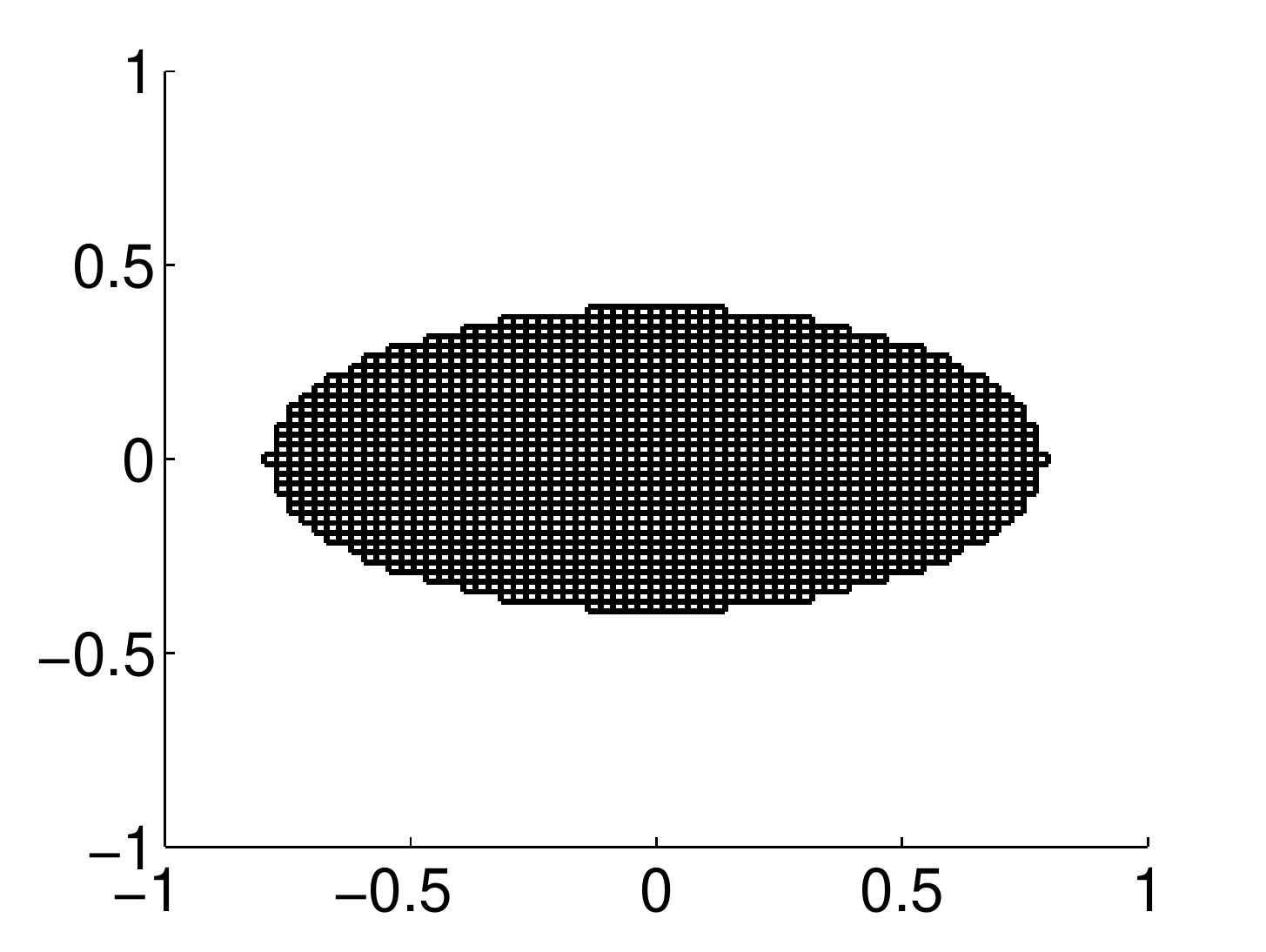}\label{fig:ellipseX}}
        \subfigure[]{\includegraphics[width=.495\textwidth]{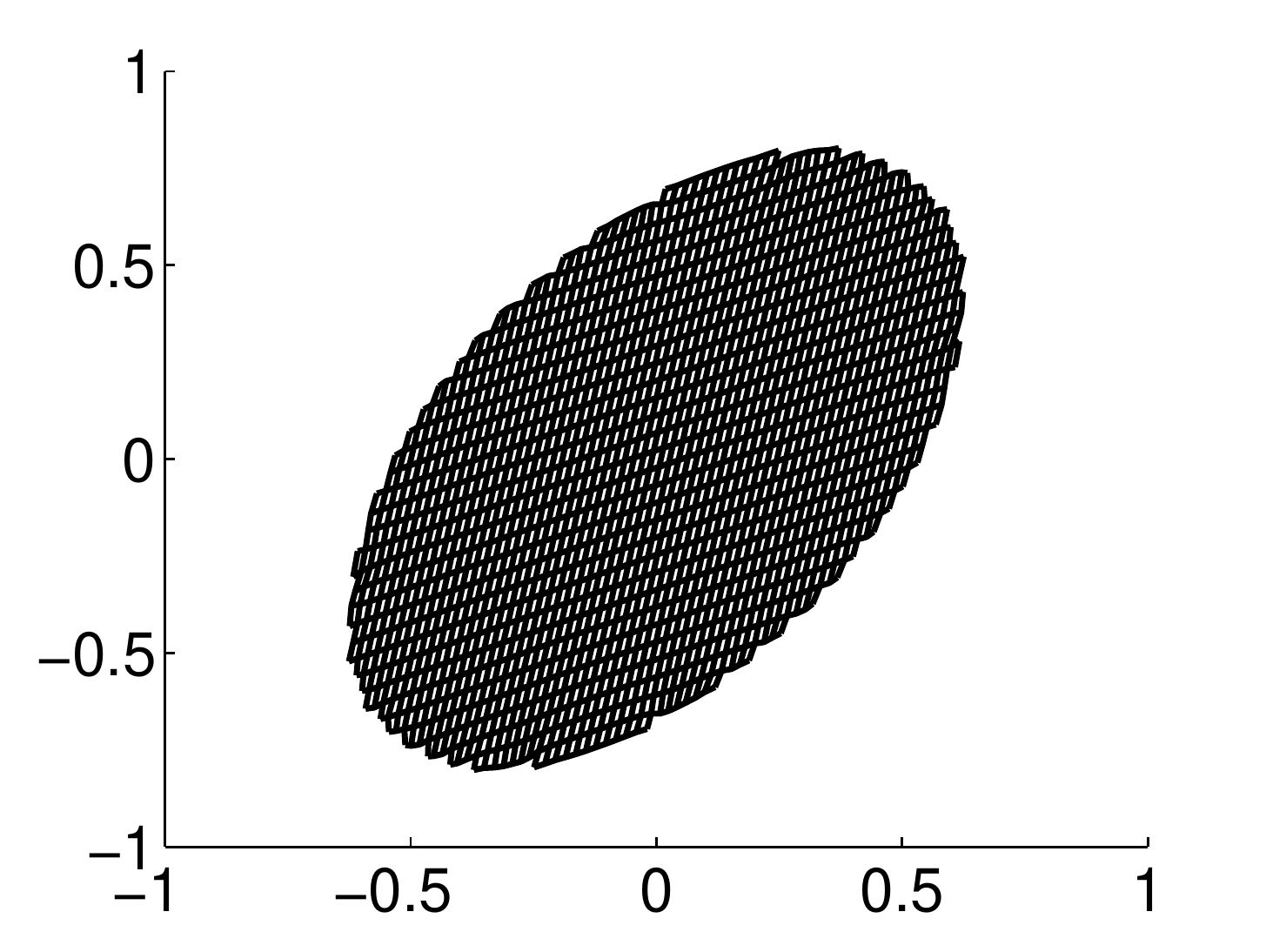}\label{fig:ellipseY}}
%	\subfigure[]{\includegraphics[width=.495\textwidth]{ellipseX2}\label{fig:ellipseX}}
 %       \subfigure[]{\includegraphics[width=.495\textwidth]{ellipseY2}\label{fig:ellipseY}}       
  	\caption{\subref{fig:ellipseX} An ellipse $\OS$ and \subref{fig:ellipseY} its image under the gradient map $\nabla u$~(\S\ref{sec:exEllipse}).}
  	\label{fig:ellipse}
\end{figure}

\begin{figure}[htdp]
	\centering
	\subfigure[]{\includegraphics[width=.49\textwidth]{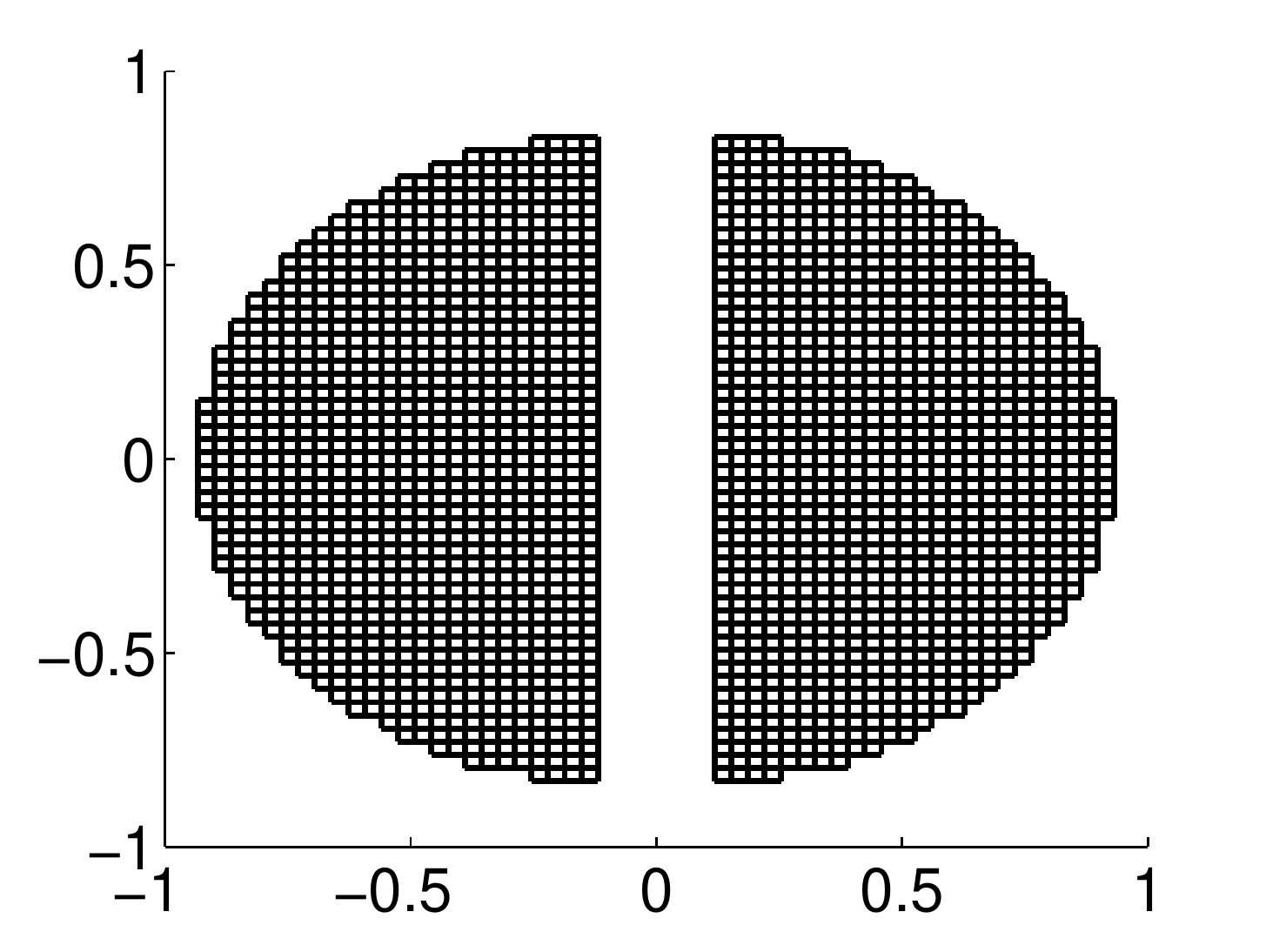}\label{fig:splitX}}
        \subfigure[]{\includegraphics[width=.49\textwidth]{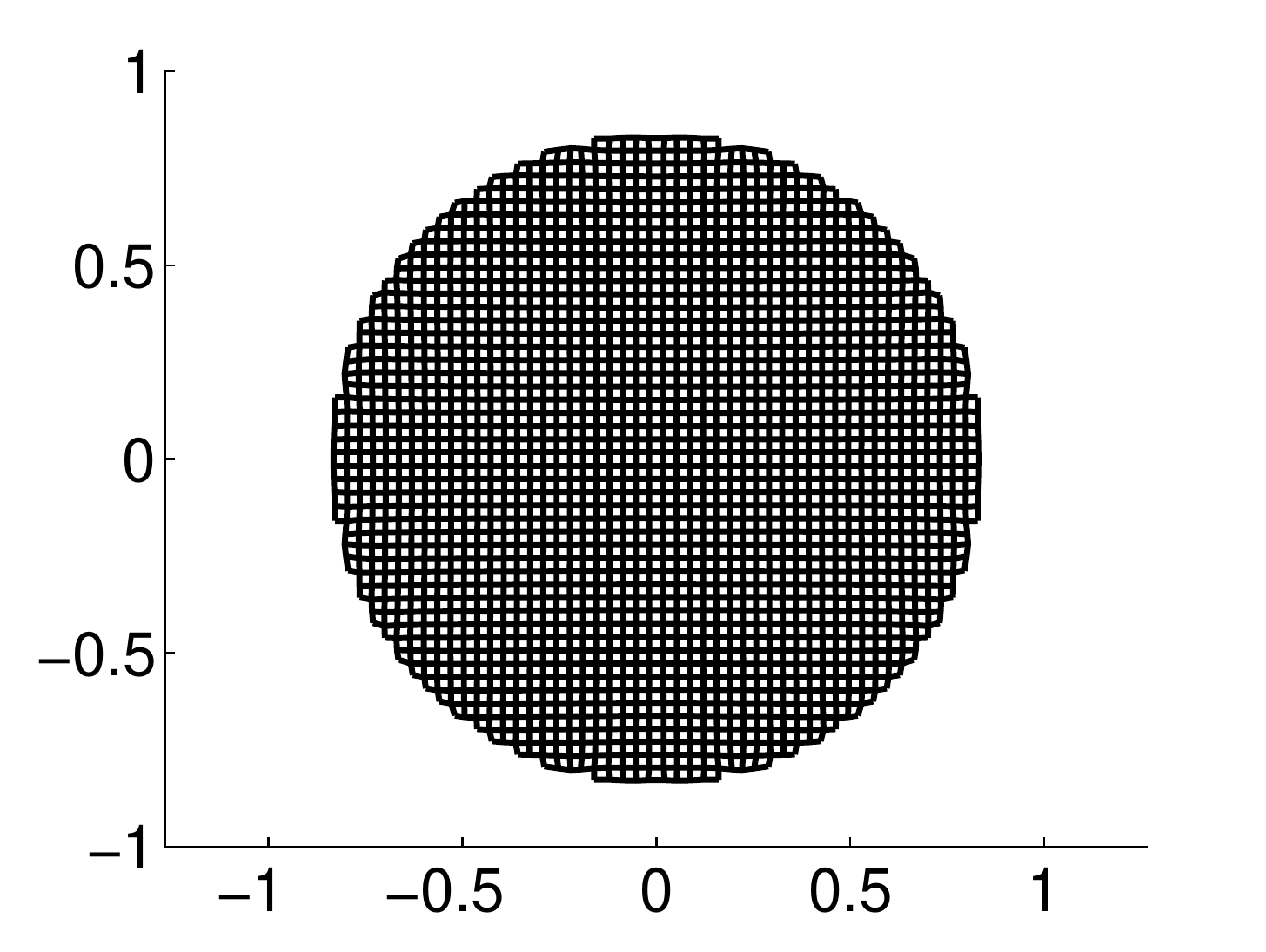}\label{fig:splitY}}
  	\caption{\subref{fig:splitX} Two half-circles $\OS$ and \subref{fig:splitY} its image under the gradient map $\nabla u$~(\S\ref{sec:exSplit}).}
  	\label{fig:split}
\end{figure}

\begin{table}[htdp]\small
\begin{center}
\begin{tabular}{c|cccccc}
 &\multicolumn{6}{c}{Maximum Error} \\
$N_X$   & \multicolumn{6}{c}{$N_Y$}  \\
    & 8 & 16 & 32 & 64 & 128 & 256\\
\hline
32 & 0.1163 & 0.0773 & 0.0693 & 0.0669 & 0.0665 & 0.0062  \\
64 & 0.1188 & 0.0403 & 0.0302 & 0.0291 & 0.0282 &  0.0283\\
128 & 0.1214 & 0.0302 & 0.0201 & 0.0174 & 0.0168 & 0.0168 \\
256 & 0.1206 & 0.0278 & 0.0116 & 0.0101 & 0.0092 &  0.0091\\
362 & 0.1175 & 0.0291 & 0.0098 & 0.0063 & 0.0057 & 0.0056 
\end{tabular}
\end{center}
\caption{Distance between exact and computed gradient maps for map from an ellipse to an ellipse. }
\label{table:ellipse}
\end{table}

\subsection{Mapping from a disconnected region}\label{sec:exSplit}
We now consider the problem of mapping the two half-circles
\begin{multline*} X = \{(x_1,x_2)\mid x_1 < -0.2,(x_1+0.2)^2+x_2^2 < 0.85^2 \} \\ \cup \{(x_1,x_2)\mid x_1 > 0.1,(x_1-0.1)^2+x_2^2 < 0.85^2 \}\end{multline*}
onto the circle
\[ Y = \{(y_1,y_2)\mid y_1^2+y_2^2<0.85^2\}.\]
This example, which is pictured in Figure~\ref{fig:split}, is degenerate since the domain is not simply connected.  Nevertheless, our method correctly computes the optimal map, as the results of \autoref{table:split} verify.

\begin{table}[htdp]\small
\begin{center}
\begin{tabular}{c|cccccc}
 &\multicolumn{6}{c}{Maximum Error} \\
$N_X$   & \multicolumn{6}{c}{$N_Y$}  \\
    & 8 & 16 & 32 & 64 & 128 & 256\\
\hline
32 & 0.0453 & 0.0267 & 0.0255 & 0.0258 & 0.0259 & 0.0258   \\
64 & 0.0397 & 0.0184 & 0.0158 & 0.0146 & 0.0144 & 0.0139  \\
128 & 0.0392 & 0.0097 & 0.0063 & 0.0066& 0.0065 & 0.0064 \\
256 & 0.0432 & 0.0110 & 0.0084 & 0.0087 & 0.0086 & 0.0073 \\
362 & 0.0448 & 0.0130 & 0.0070 & 0.0047 & 0.0045 & 0.0039 
\end{tabular}
\end{center}
\caption{Distance between exact and computed gradient maps for map from two semi-circles to a circle.  }
\label{table:split}
\end{table}

\subsection{Other geometries} 
To give a flavor of the types of geometry that can be captured by solving~\eqref{OT1} with our approximation scheme, we also provide several different computed maps in \autoref{fig:examples}.  
These were all obtained by mapping a square with uniform density onto a specified convex set, whose boundary could consist of a combination of straight and curved edges.  While no exact solutions are available, we can certainly verify that the computed gradient does map the square into the correct target set.

\begin{figure}[htdp]
	\centering
%	\subfigure[]{\includegraphics[width=.32\textwidth]{circle}\label{fig:circle}}
%  \subfigure[]{\includegraphics[width=.32\textwidth]{triangle}\label{fig:triangle}}
%  \subfigure[]{\includegraphics[width=.32\textwidth]{diamond}\label{fig:diamond}}
%  \subfigure[]{\includegraphics[width=.55\textwidth]{bowl}\label{fig:bowl}}
%  \subfigure[]{\includegraphics[width=.4\textwidth]{icecream}\label{fig:icecream}}
 	\subfigure[]{\includegraphics[height=.5\textwidth]{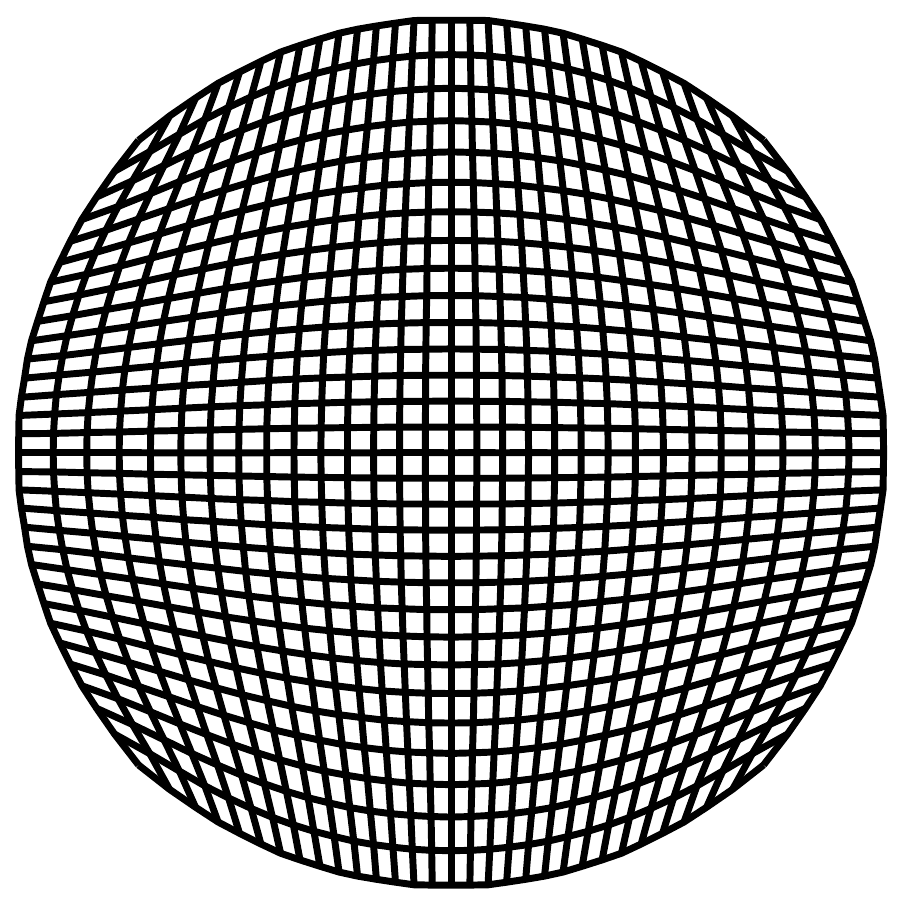}\label{fig:circle}}
  \subfigure[]{\includegraphics[height=.45\textwidth]{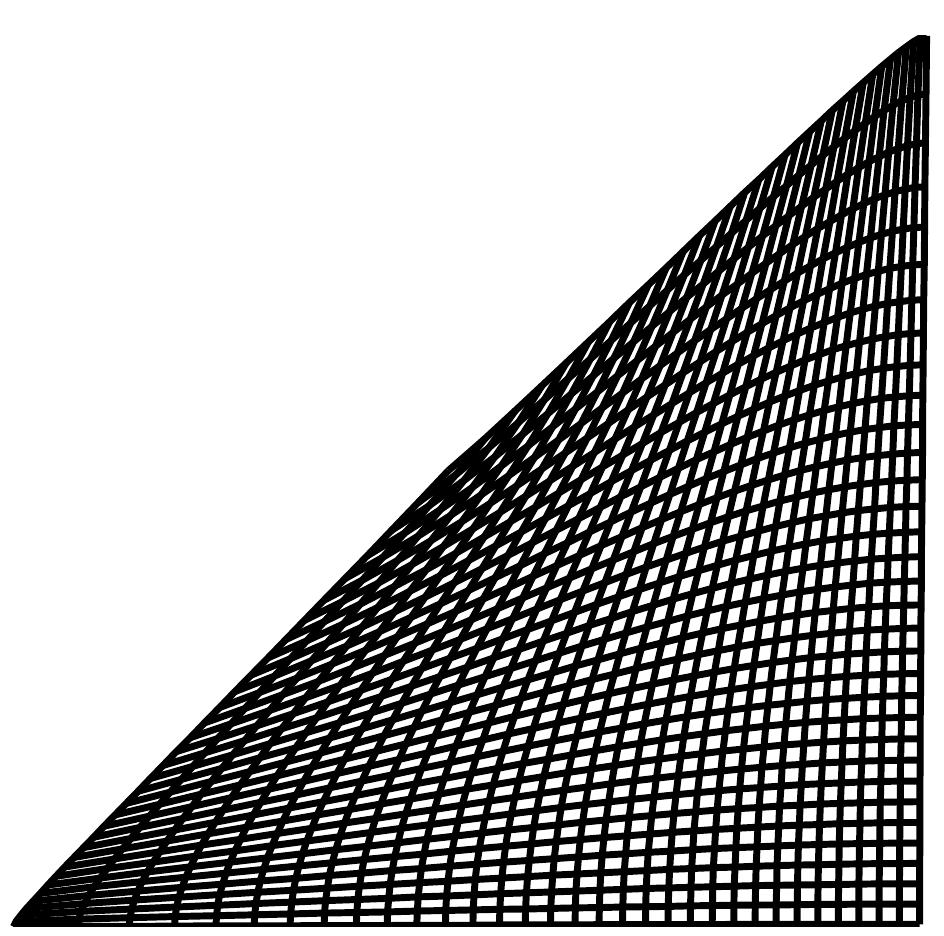}\label{fig:triangle}}
  \subfigure[]{\includegraphics[height=.5\textwidth]{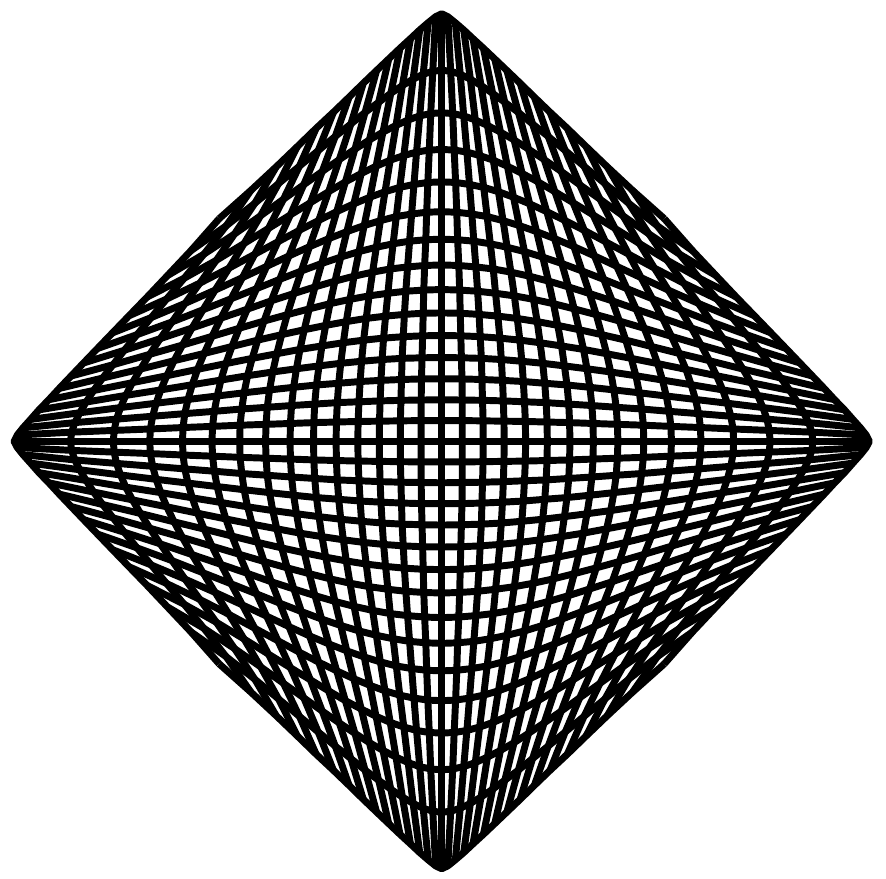}\label{fig:diamond}}
  \subfigure[]{\includegraphics[height =.5\textwidth]{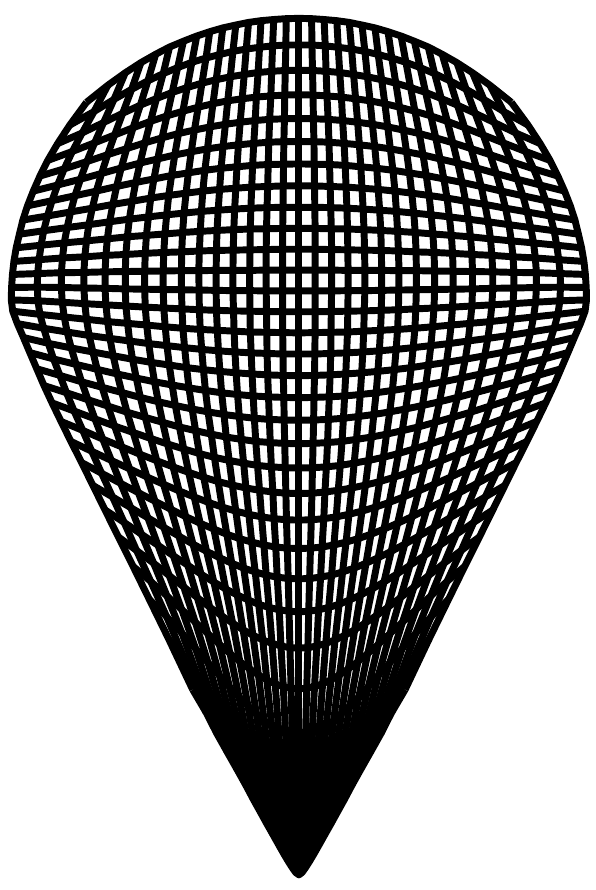}\label{fig:icecream}}
  \subfigure[]{\includegraphics[height=.35\textwidth]{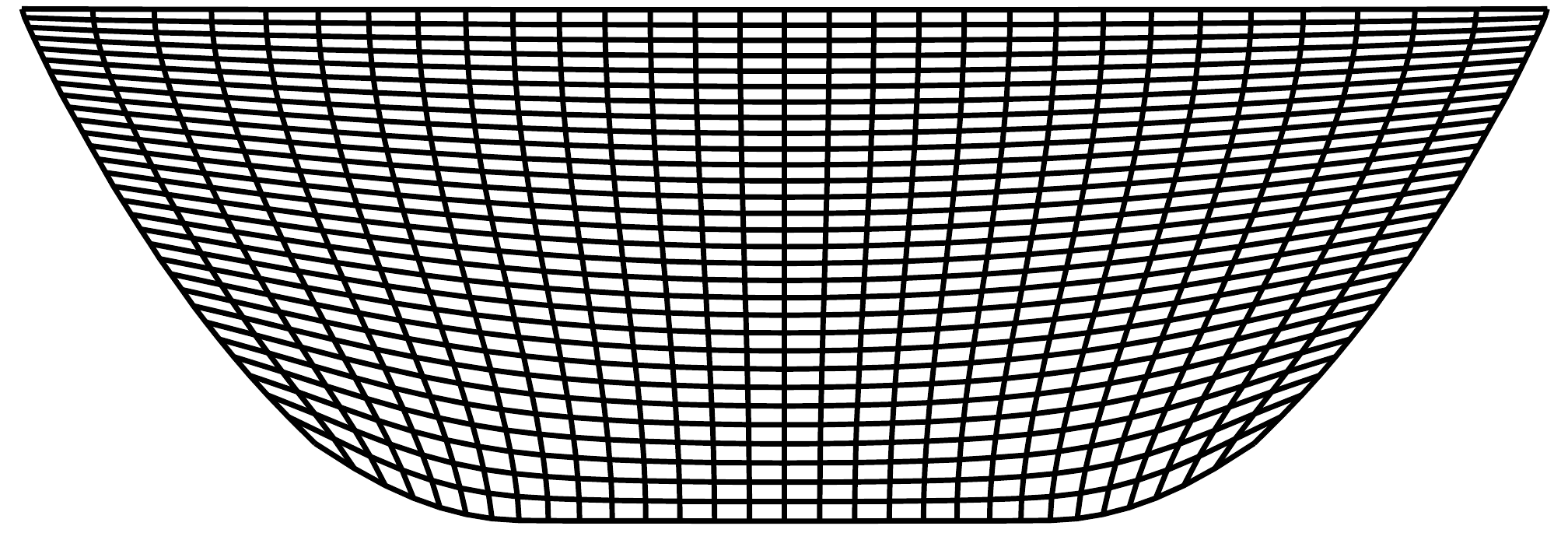}\label{fig:bowl}}
 
  	\caption{Computed maps from the uniform density on a standard orientation square to the uniform density on the illustrated target domains.}
  	\label{fig:examples}
\end{figure} 

\section{Conclusions}
A rigorous approach to finding solutions of the Optimal Transportation problem was presented, using convergent finite difference approximations to the elliptic \MA PDE with 
corresponding boundary conditions.

The target domain was represented using the signed distance function to the boundary.  This representation led to a Hamilton-Jacobi equation for the optimal transportation boundary conditions.  This Hamilton-Jacobi equation was discretized using  a monotone finite difference scheme, which relied on obliqueness conditions for the solution to build a scheme using only values of the function inside the source domain.  

The combined PDE and boundary conditions led to an nearly elliptic finite difference scheme, which was proved to converge to the unique viscosity solution of the underlying PDE. 
The convergence proof  used the theory of filtered schemes,~\cite{ObermanFroeseFiltered}, which is an extension of the monotone schemes in the Barles-Souganidis theory~\cite{BSnum}.

Computational results were presented which are consistent with the predictions of the theory.  A complete numerical implementation of the scheme, including a fast Newton solver, is presented in the companion paper~\cite{SBVP_Num}.
 
A first application is to use this method to compute JKO gradient flows~\cite{MR1617171,MR1842429}.  Another application is to OT problem with more general cost functions $c(x-y)$, using a version of the Monge-Amp\`ere PDE for $c$-convex potentials.

\bibliographystyle{plain}
\bibliography{SBVPT}

\end{document}